\documentclass{amsart}

\newtheorem{theorem}{Theorem}[section]
\newtheorem{lemma}[theorem]{Lemma}
\newtheorem{corollary}[theorem]{Corollary}

\theoremstyle{definition}

\theoremstyle{remark}
\newtheorem{remark}[theorem]{Remark}
\usepackage{amsmath}
\usepackage{amssymb}
\usepackage{graphicx}
\usepackage{mathrsfs}
\usepackage{algorithm}
\usepackage{algorithmicx}
\usepackage{appendix}
\usepackage{enumitem}
\usepackage{stmaryrd}
\usepackage{pifont}
\usepackage{bm}
\usepackage{booktabs}
\usepackage{array}
\usepackage{multirow}
\usepackage[colorlinks, linkcolor=blue, anchorcolor=blue, citecolor=blue]{hyperref}
\AtBeginDocument{
                 \label{CorrectFirstPageLabel}
                 
                }

\DeclareMathOperator*{\rank}{rank}
\DeclareMathOperator*{\diag}{diag}

\numberwithin{equation}{section}

\begin{document}

\title[TG convergence theory for SPSD systems]{Convergence analysis of two-grid methods for symmetric positive semidefinite systems}

%    Information for first author
\author{Xuefeng Xu}
%    Address of record for the research reported here
\address{School of Mathematics, Southeast University, Nanjing 211189, China}
\email{xuxuefeng@lsec.cc.ac.cn; xuxuefeng@seu.edu.cn}
%    \thanks will become a 1st page footnote.
%\thanks{}

%    Information for second author
%\author{}
%\address{}
%\email{}
%\thanks{}

%    General info
\subjclass[2020]{Primary 65F08, 65F10, 65N55; Secondary 15A10, 15A18}

%\date{January 1, 2001 and, in revised form, June 22, 2001.}

%\dedicatory{This paper is dedicated to our advisors.}

\keywords{multigrid methods, two-grid methods, symmetric positive semidefinite matrices, null space, Moore--Penrose inverse}

\begin{abstract}
Two-grid theory plays a fundamental role in the design and analysis of multigrid methods. This paper is devoted to a new convergence analysis of two-grid methods for singular and symmetric positive semidefinite systems. Specifically, we derive a concise identity for characterizing the convergence factor of two-grid methods, with the Moore--Penrose inverse of coarse-grid matrix being used as a coarse solver. Furthermore, we present a convergence estimate for two-grid methods with approximate coarse solvers. Our new theory does not require any additional assumptions on the coefficient matrix, especially on its null space.
\end{abstract}

\maketitle

%\section*{This is an unnumbered first-level section head}
%This is an example of an unnumbered first-level heading.

%% The correct journal style for \specialsection is all uppercase; a known bug
%% in amsart.cls prevents this, so input must be uppercase until it is fixed.
%\specialsection*{This is a Special Section Head}
%\specialsection*{THIS IS A SPECIAL SECTION HEAD}
%This is an example of a special section head%
%%%%%%%%%%%%%%%%%%%%%%%%%%%%%%%%%%%%%%%%%%%%%%%%%%%%%%%%%%%%%%%%%%%%%%%%
%\footnote{Here is an example of a footnote. Notice that this footnote
%text is running on so that it can stand as an example of how a footnote
%with separate paragraphs should be written.
%\par
%And here is the beginning of the second paragraph.}%
%%%%%%%%%%%%%%%%%%%%%%%%%%%%%%%%%%%%%%%%%%%%%%%%%%%%%%%%%%%%%%%%%%%%%%%%

\section{Introduction}

Multigrid is a powerful solver with linear or near-linear computational complexity that has been widely used for the numerical solution of partial differential equations; see, e.g.,~\cite{Hackbusch1985,Briggs2000,Trottenberg2001,Vassilevski2008,XZ2017}. The foundation of multigrid is a two-grid scheme, which consists of two complementary error-reduction processes: \textit{smoothing} (or \textit{relaxation}) and \textit{coarse-grid correction}. In general, the smoothing process is chosen as a simple iterative method (e.g., the weighted Jacobi and Gauss--Seidel methods), while these classical methods are commonly only effective at reducing high-frequency (or oscillatory) error. The remaining low-frequency (or smooth) error will be treated by the correction process, in which a coarse-grid system with fewer unknowns needs to be solved. In practice, it is often too costly to solve the coarse-grid system exactly. A typical strategy is to apply the two-grid scheme recursively in the correction steps. The resulting methods (namely, multigrid methods) possess a multilevel structure.

For symmetric positive definite (SPD) linear systems, the convergence of multigrid methods has been well studied; see, e.g.,~\cite{Xu1992,Yserentant1993,XZ2002,Falgout2004,Falgout2005,MacLachlan2014,XZ2017,XXF2018,XXF2022-1,XXF2022-2,XXF2025}. Many important models or methods in mathematics may lead to symmetric positive semidefinite (SPSD) linear systems, such as partial differential equations with pure Neumann boundary conditions (see, e.g.,~\cite{Bochev2005}), generalized finite element methods (see, e.g.,~\cite{Strouboulis2000-1,Strouboulis2000-2,Strouboulis2001,XZ2003}), and Laplacian matrices of graphs (see, e.g.,~\cite{Merris1994,Spielman2010,Livne2012}). Convergence of additive and multiplicative Schwarz methods for SPSD systems was studied in~\cite{Nabben2006}. Convergence theory of general stationary iterative methods for such systems can be found, e.g., in~\cite{Keller1965,Cao2001,Lee2006,Cao2008,Frommer2008,Wu2008,Frommer2014}.

Our focus here is on two-grid methods for SPSD systems. Two-grid theory plays a fundamental role in the design or analysis of multigrid methods; see, e.g.,~\cite{Brannick2010,Manteuffel2018,Manteuffel2019,XXF2022-1,XXF2022-2,XXF2025}. In the SPD case, an elegant identity for characterizing the convergence factor of two-grid methods was established in~\cite{Falgout2005}; see~\cite{XZ2002,Zikatanov2008} for an abstract version. The identity is a powerful tool for analyzing the properties of two-grid methods; see, e.g.,~\cite{Falgout2005,Brannick2018,XXF2018}. Similarly, we expect to establish a convergence identity for two-grid methods applied to SPSD systems, without requiring any additional conditions on the null space of the coefficient matrix $A$, denoted by $\mathcal{N}(A)$.

Under several essential conditions on local null spaces, a sharp convergence estimate for successive subspace correction methods~\cite{Xu1992,XZ2002} applied to SPSD systems was presented in~\cite{Lee2008}; see~\cite{Wu2008} for a simplified analysis. More specifically, \cite{Lee2008} offers an identity for the seminorm of error propagation operator, which involves a `sup-inf-inf' expression relating to the global null space and its orthogonal complement. In~\cite{Bolten2011}, a classical two-grid convergence estimate for SPD systems (see~\cite{McCormick1982,Ruge1987,Stuben2001}) was extended to the semidefinite case. A two-grid convergence theory for singular systems was developed in~\cite{Notay2016}, in which a condition on the coarse space and $\mathcal{N}(A)$ is required. Moreover, an extended prolongation operator, constructed from $\mathcal{N}(A)$, is needed. A noteworthy contribution is the two-grid identity in~\cite[Theorem~5.3]{XZ2017}, which was proved under the condition that $\mathcal{N}(A)$ is contained in the coarse space.

In some cases, however, the \textit{full} information of $\mathcal{N}(A)$ may not be readily available. For example, as indicated in~\cite[page~52]{Strouboulis2000-1}, the eigenfunctions associated with the zero eigenvalues are generally unknown in generalized finite element methods. In addition, most of the existing results focus on the case where the coarse-grid system is solved exactly, in the sense that a generalized inverse of the coarse-grid matrix is used as a coarse solver. For singular systems, more general coarse solvers are seldom considered in the literature. The difficulty may arise from the fact that the error propagation operator of coarse-grid correction fails to be a projector when the coarse-grid system is solved approximately.

In this paper, we establish a new two-grid convergence theory for SPSD systems. Our main contribution is twofold:
\begin{itemize}[leftmargin=0.85cm]

\item We derive a succinct identity for the convergence factor of two-grid methods, with the Moore--Penrose inverse of coarse-grid matrix being used as a coarse solver (for convenience, it is called an \textit{exact} coarse solver);

\item We present a convergence estimate for two-grid methods with inexact coarse solvers (various approximation strategies can be utilized to solve the coarse-grid system).

\end{itemize}
It is worth mentioning that our theory has several distinctive features in comparison with the existing ones, as described below.
\begin{itemize}[leftmargin=0.85cm]

\item It does not require any additional conditions on the coefficient matrix, especially on its null space.

\item It gives an identity for the convergence factor of exact two-grid methods instead of an upper bound.

\item It provides a convergence estimate for two-grid methods with inexact coarse solvers, that is, our analysis is not confined to the exact case.

\end{itemize}

The rest of this paper is organized as follows. In Section~\ref{sec:pre}, we give a sufficient and necessary condition for two-grid convergence factor to be less than $1$. In Section~\ref{sec:exact}, we present a concise identity for the convergence factor of exact two-grid methods. In Section~\ref{sec:inexact}, we study the convergence of two-grid methods with approximate coarse solvers, and provide a numerical example to illustrate our theoretical estimate. In Section~\ref{sec:con}, we make some concluding remarks.

\section{Preliminaries} \label{sec:pre}

We start with some notation that will be used in the subsequent discussions.

\begin{itemize}[leftmargin=0.9cm]

\item[--] For a symmetric matrix $S$, $S\succ 0$ and $S\succeq 0$ mean that $S$ is positive definite and positive semidefinite, respectively.

\item[--] $I_{n}$ denotes the $n\times n$ identity matrix (or $I$ when its size is clear from context).

\item[--] $\mathcal{R}(\cdot)$ denotes the column space (or range) of a matrix.

\item[--] $\mathcal{N}(\cdot)$ denotes the null space (or kernel) of a matrix.

\item[--] $\lambda_{\min}(\cdot)$ and $\lambda_{\max}(\cdot)$ denote the smallest and largest eigenvalues of a matrix with real eigenvalues, respectively.

\item[--] $\lambda_{i}(\cdot)$ denotes the $i$th smallest eigenvalue of a matrix with real eigenvalues.

\item[--] $\lambda(\cdot)$ denotes the spectrum of a matrix.

\item[--] $\|\cdot\|_{2}$ denotes the Euclidean norm of a vector.

\item[--] $(\cdot)^{\dagger}$ denotes the Moore--Penrose inverse of a matrix.

\end{itemize}

For any $A\in\mathbb{R}^{n\times n}\backslash\{0\}$ and $A\succeq 0$, we define
\begin{displaymath}
\|\mathbf{v}\|_{A}:=\sqrt{\mathbf{v}^{T}A\mathbf{v}} \quad \forall\,\mathbf{v}\in\mathbb{R}^{n}
\end{displaymath}
and
\begin{equation}\label{A-seminorm}
\|X\|_{A}:=\sup_{\mathbf{v}\in\mathbb{R}^{n}\backslash\mathcal{N}(A)}\frac{\|X\mathbf{v}\|_{A}}{\|\mathbf{v}\|_{A}} \quad \forall\,X\in\mathbb{R}^{n\times n}.
\end{equation}
It is easy to check that the function $\|\cdot\|_{A}:\mathbb{R}^{n\times n}\rightarrow[0,+\infty]$ satisfies the following properties:
\begin{itemize}[leftmargin=1.1cm]

\item[(a)] $\|X\|_{A}\geq 0$ for all $X\in\mathbb{R}^{n\times n}$;

\item[(b)] $\|\alpha X\|_{A}=|\alpha|\|X\|_{A}$ for all $\alpha\in\mathbb{R}$ and $X\in\mathbb{R}^{n\times n}$;

\item[(c)] $\|X_{1}+X_{2}\|_{A}\leq\|X_{1}\|_{A}+\|X_{2}\|_{A}$ for all $X_{1},X_{2}\in\mathbb{R}^{n\times n}$.

\end{itemize}
That is, \eqref{A-seminorm} defines a \textit{seminorm} on $\mathbb{R}^{n\times n}$.

The following lemma gives two elementary but useful properties of the $A$-seminorm.

\begin{lemma}\label{lem:norm}
If $X\in\mathbb{R}^{n\times n}$ and $A^{\frac{1}{2}}X=YA^{\frac{1}{2}}$ for some $Y\in\mathbb{R}^{n\times n}$, then
\begin{equation}\label{mat-vec}
\|X\mathbf{v}\|_{A}\leq\|X\|_{A}\|\mathbf{v}\|_{A} \quad \forall\,\mathbf{v}\in\mathbb{R}^{n}
\end{equation}
and
\begin{equation}\label{A-norm}
\|X\|_{A}=\max_{\mathbf{v}\in\mathcal{R}(A)\backslash\{0\}}\frac{\|Y\mathbf{v}\|_{2}}{\|\mathbf{v}\|_{2}}.
\end{equation}
\end{lemma}

\begin{proof}
Note that
\begin{displaymath}
\mathcal{N}(A)=\mathcal{N}\big(A^{\frac{1}{2}}\big) \quad \text{and} \quad \mathcal{R}(A)=\mathcal{R}\big(A^{\frac{1}{2}}\big).
\end{displaymath}
It can be readily seen from~\eqref{A-seminorm} that the inequality~\eqref{mat-vec} holds for $\mathbf{v}\in\mathbb{R}^{n}\backslash\mathcal{N}(A)$. If $\mathbf{v}\in\mathcal{N}(A)$, then
\begin{displaymath}
\|X\mathbf{v}\|_{A}=\big\|A^{\frac{1}{2}}X\mathbf{v}\big\|_{2}=\big\|YA^{\frac{1}{2}}\mathbf{v}\big\|_{2}=0,
\end{displaymath}
where we have used the condition $A^{\frac{1}{2}}X=YA^{\frac{1}{2}}$. Hence, the inequality~\eqref{mat-vec} holds for all $\mathbf{v}\in\mathbb{R}^{n}$.

From~\eqref{A-seminorm}, we have
\begin{align*}
\|X\|_{A}&=\sup_{\mathbf{v}\in\mathbb{R}^{n}\backslash\mathcal{N}(A)}\frac{\big\|A^{\frac{1}{2}}X\mathbf{v}\big\|_{2}}{\big\|A^{\frac{1}{2}}\mathbf{v}\big\|_{2}}=\sup_{\mathbf{v}\in\mathbb{R}^{n}\backslash\mathcal{N}(A)}\frac{\big\|YA^{\frac{1}{2}}\mathbf{v}\big\|_{2}}{\big\|A^{\frac{1}{2}}\mathbf{v}\big\|_{2}}\\
&=\sup_{\mathbf{v}\in\mathcal{R}(A)\backslash\{0\}}\frac{\|Y\mathbf{v}\|_{2}}{\|\mathbf{v}\|_{2}}=\max_{\mathbf{v}\in\mathcal{R}(A)\backslash\{0\}}\frac{\|Y\mathbf{v}\|_{2}}{\|\mathbf{v}\|_{2}},
\end{align*}
which yields the expression~\eqref{A-norm}.
\end{proof}

Consider solving the linear system
\begin{equation}\label{system}
A\mathbf{u}=\mathbf{f},
\end{equation}
where $A\in\mathbb{R}^{n\times n}\backslash\{0\}$ is SPSD, $\mathbf{u}\in\mathbb{R}^{n}$, and $\mathbf{f}\in\mathcal{R}(A)$ (that is, $\mathbf{f}$ is orthogonal to $\mathcal{N}(A)$). When $A$ is singular (i.e., $\mathcal{N}(A)\neq\{0\}$), the system~\eqref{system} has infinitely many solutions for any $\mathbf{f}\in\mathcal{R}(A)$. We will analyze the convergence of two-grid methods for solving~\eqref{system}. Two basic assumptions involved in the analysis of two-grid methods are made below.
\begin{itemize}[leftmargin=1.5cm]

\item[\textbf{(A1):}] Let $M$ be an $n\times n$ matrix such that $\|I-MA\|_{A}\leq 1$.

\item[\textbf{(A2):}] Let $P$ be an $n\times n_{\rm c}$ prolongation (or interpolation) matrix such that the Galerkin coarse-grid matrix $A_{\rm c}:=P^{T}AP$ is nonzero, where $n_{\rm c}\,(<n)$ is the number of coarse variables.

\end{itemize}

\begin{remark}
In view of \textbf{(A1)} and \textbf{(A2)}, $M$ is not necessarily nonsingular, and $P$ is not necessarily of full column rank. In addition, measured in $\|\cdot\|_{A}$, the smoothing iteration in Algorithm~\ref{alg:TG} below is not required to be a strict contraction. Due to
\begin{displaymath}
A^{\frac{1}{2}}(I-MA)=\big(I-A^{\frac{1}{2}}MA^{\frac{1}{2}}\big)A^{\frac{1}{2}},
\end{displaymath}
it follows from~\eqref{A-norm} that
\begin{displaymath}
\|I-MA\|_{A}^{2}=\max_{\mathbf{v}\in\mathcal{R}(A)\backslash\{0\}}\frac{\big\|\big(I-A^{\frac{1}{2}}MA^{\frac{1}{2}}\big)\mathbf{v}\big\|_{2}^{2}}{\|\mathbf{v}\|_{2}^{2}}=1-\min_{\mathbf{v}\in\mathcal{R}(A)\backslash\{0\}}\frac{\mathbf{v}^{T}A^{\frac{1}{2}}\overline{M}A^{\frac{1}{2}}\mathbf{v}}{\mathbf{v}^{T}\mathbf{v}},
\end{displaymath}
where
\begin{equation}\label{barM}
\overline{M}:=M+M^{T}-M^{T}AM.
\end{equation}
Thus, $\|I-MA\|_{A}\leq 1$ if and only if $\mathbf{v}^{T}A^{\frac{1}{2}}\overline{M}A^{\frac{1}{2}}\mathbf{v}\geq 0$ for all $\mathbf{v}\in\mathcal{R}(A)\backslash\{0\}$, or, equivalently, $\mathbf{v}^{T}A^{\frac{1}{2}}\overline{M}A^{\frac{1}{2}}\mathbf{v}\geq 0$ for all $\mathbf{v}\in\mathbb{R}^{n}\backslash\{0\}$ (i.e., $A^{\frac{1}{2}}\overline{M}A^{\frac{1}{2}}\succeq 0$), because $\mathbb{R}^{n}$ is the direct sum of $\mathcal{R}(A)$ and $\mathcal{N}(A)$, denoted by $\mathbb{R}^{n}=\mathcal{R}(A)\oplus\mathcal{N}(A)$.
\end{remark}

With the above assumptions, a two-grid method for solving~\eqref{system} is described by Algorithm~\ref{alg:TG}, in which the Moore--Penrose inverse $A_{\rm c}^{\dagger}$ is used as a coarse solver. For convenience, such an algorithm is called an \textit{exact} two-grid method.

\begin{algorithm}[!htbp]

\caption{\ Exact two-grid method.}\label{alg:TG}

\smallskip

\begin{algorithmic}[1]

\State Smoothing: $\mathbf{u}^{(1)}\gets\mathbf{u}^{(0)}+M\big(\mathbf{f}-A\mathbf{u}^{(0)}\big)$ \Comment{$\mathbf{u}^{(0)}\in\mathbb{R}^{n}$ is an initial guess}

\smallskip

\State Restriction: $\mathbf{r}_{\rm c}\gets P^{T}\big(\mathbf{f}-A\mathbf{u}^{(1)}\big)$

\smallskip

\State Coarse-grid correction: $\mathbf{e}_{\rm c}\gets A_{\rm c}^{\dagger}\mathbf{r}_{\rm c}$

\smallskip

\State Prolongation: $\mathbf{u}_{\rm TG}\gets\mathbf{u}^{(1)}+P\mathbf{e}_{\rm c}$

\smallskip

\end{algorithmic}

\end{algorithm}

\begin{remark}
In Algorithm~\ref{alg:TG}, the coarse-grid system to be solved reads	
\begin{equation}\label{coarse-system}
A_{\rm c}\mathbf{x}_{\rm c}=\mathbf{r}_{\rm c}.
\end{equation}
It is easy to see that $\mathbf{e}_{\rm c}=A_{\rm c}^{\dagger}\mathbf{r}_{\rm c}$ is a solution to~\eqref{coarse-system}, because $A_{\rm c}A_{\rm c}^{\dagger}$ is a projector onto $\mathcal{R}(A_{\rm c})$ and
\begin{displaymath}
\mathbf{r}_{\rm c}=P^{T}\big(\mathbf{f}-A\mathbf{u}^{(1)}\big)=P^{T}A\big(\mathbf{u}-\mathbf{u}^{(1)}\big)\in\mathcal{R}\big(P^{T}A^{\frac{1}{2}}\big)=\mathcal{R}(A_{\rm c}).
\end{displaymath}
\end{remark}

From Algorithm~\ref{alg:TG}, we have
\begin{equation}\label{iteration}
\mathbf{u}-\mathbf{u}_{\rm TG}=E_{\rm TG}\big(\mathbf{u}-\mathbf{u}^{(0)}\big),
\end{equation}
where
\begin{displaymath}
E_{\rm TG}=\big(I-PA_{\rm c}^{\dagger}P^{T}A\big)(I-MA).
\end{displaymath}
Define
\begin{equation}\label{pi}
\Pi:=A^{\frac{1}{2}}PA_{\rm c}^{\dagger}P^{T}A^{\frac{1}{2}}.
\end{equation}
It can be easily verified that
\begin{displaymath}
\Pi^{2}=\Pi=\Pi^{T}, \quad \mathcal{N}(\Pi)=\mathcal{N}\big(P^{T}A^{\frac{1}{2}}\big), \quad \text{and} \quad \mathcal{R}(\Pi)=\mathcal{R}\big(A^{\frac{1}{2}}P\big),
\end{displaymath}
i.e., $\Pi$ is an $L^{2}$-orthogonal projector along (or parallel to) $\mathcal{N}\big(P^{T}A^{\frac{1}{2}}\big)$ onto $\mathcal{R}\big(A^{\frac{1}{2}}P\big)$. Since
\begin{displaymath}
A^{\frac{1}{2}}E_{\rm TG}=(I-\Pi)\big(I-A^{\frac{1}{2}}MA^{\frac{1}{2}}\big)A^{\frac{1}{2}},
\end{displaymath}
it follows from Lemma~\ref{lem:norm} and~\eqref{iteration} that
\begin{displaymath}
\|\mathbf{u}-\mathbf{u}_{\rm TG}\|_{A}\leq\|E_{\rm TG}\|_{A}\|\mathbf{u}-\mathbf{u}^{(0)}\|_{A},
\end{displaymath}
where $\|E_{\rm TG}\|_{A}$, called the \textit{convergence factor} of Algorithm~\ref{alg:TG}, is given by
\begin{equation}\label{ETG-Anorm}
\|E_{\rm TG}\|_{A}=\max_{\mathbf{v}\in\mathcal{R}(A)\backslash\{0\}}\frac{\big\|(I-\Pi)\big(I-A^{\frac{1}{2}}MA^{\frac{1}{2}}\big)\mathbf{v}\big\|_{2}}{\|\mathbf{v}\|_{2}}.
\end{equation}
Then
\begin{align*}
\|E_{\rm TG}\|_{A}^{2}&=\max_{\mathbf{v}\in\mathcal{R}(A)\backslash\{0\}}\frac{\mathbf{v}^{T}\big(I-A^{\frac{1}{2}}M^{T}A^{\frac{1}{2}}\big)(I-\Pi)\big(I-A^{\frac{1}{2}}MA^{\frac{1}{2}}\big)\mathbf{v}}{\mathbf{v}^{T}\mathbf{v}}\\
&=1-\min_{\mathbf{v}\in\mathcal{R}(A)\backslash\{0\}}\frac{\mathbf{v}^{T}F_{\rm TG}\mathbf{v}}{\mathbf{v}^{T}\mathbf{v}},
\end{align*}
where
\begin{equation}\label{F-TG}
F_{\rm TG}=A^{\frac{1}{2}}\overline{M}A^{\frac{1}{2}}+\big(I-A^{\frac{1}{2}}M^{T}A^{\frac{1}{2}}\big)\Pi\big(I-A^{\frac{1}{2}}MA^{\frac{1}{2}}\big).
\end{equation}
Hence, it holds that $\|E_{\rm TG}\|_{A}<1$ if and only if $\mathbf{v}^{T}F_{\rm TG}\mathbf{v}>0$ for all $\mathbf{v}\in\mathcal{R}(A)\backslash\{0\}$.

\begin{remark}
If $\|E_{\rm TG}\|_{A}<1$ and Algorithm~\ref{alg:TG} is carried out $\nu$ times iteratively (the corresponding output is denoted by $\mathbf{u}_{\rm TG}^{(\nu)}$), then $\|\mathbf{u}-\mathbf{u}_{\rm TG}^{(\nu)}\|_{A}\to 0$ as $\nu\to\infty$, in which case $\big\{\mathbf{u}_{\rm TG}^{(\nu)}\big\}_{\nu=1}^{\infty}$ converge to a particular solution to~\eqref{system}. Indeed, $\|\mathbf{u}-\mathbf{u}_{\rm TG}^{(\nu)}\|_{A}=0$ implies that $\mathbf{u}-\mathbf{u}_{\rm TG}^{(\nu)}\in\mathcal{N}(A)$ and hence $A\mathbf{u}_{\rm TG}^{(\nu)}=A\mathbf{u}=\mathbf{f}$.
\end{remark}

The following lemma provides a sufficient and necessary condition (namely,~\eqref{equiv-cond}) for $\|E_{\rm TG}\|_{A}<1$. It is worth pointing out that~\eqref{equiv-cond} is \textbf{not really} a restriction on $\mathcal{N}(A)$. In fact, we can deduce from~\eqref{equiv-cond} that a sufficient condition for $\|E_{\rm TG}\|_{A}<1$ is $\overline{M}\succ 0$. This condition can be satisfied by the conventional weighted Jacobi and Gauss--Seidel smoothers; see Remark~\ref{rek:SPD-M} below. \textit{Throughout this paper, we always assume that {\rm\textbf{(A1)}}, {\rm\textbf{(A2)}}, and~\eqref{equiv-cond} hold}.

\begin{lemma}\label{lem:equiv-cond}
The convergence factor $\|E_{\rm TG}\|_{A}$ is strictly less than $1$ if and only if
\begin{equation}\label{equiv-cond}
\mathcal{N}\big(A^{\frac{1}{2}}\overline{M}A^{\frac{1}{2}}\big)\cap\mathcal{N}\big(P^{T}(I-AM)A^{\frac{1}{2}}\big)=\mathcal{N}(A),
\end{equation}
where $\overline{M}$ is defined by~\eqref{barM}.
\end{lemma}

\begin{proof}
Observe first that
\begin{displaymath}
\mathbf{v}^{T}F_{\rm TG}\mathbf{v}=\underbrace{\mathbf{v}^{T}A^{\frac{1}{2}}\overline{M}A^{\frac{1}{2}}\mathbf{v}}_{\geq\,0}+\underbrace{\mathbf{v}^{T}\big(I-A^{\frac{1}{2}}M^{T}A^{\frac{1}{2}}\big)\Pi\big(I-A^{\frac{1}{2}}MA^{\frac{1}{2}}\big)\mathbf{v}}_{\geq\,0}\geq 0 \quad \forall\,\mathbf{v}\in\mathbb{R}^{n}
\end{displaymath}
and
\begin{displaymath}
\mathcal{N}(A)\subseteq\mathcal{N}\big(A^{\frac{1}{2}}\overline{M}A^{\frac{1}{2}}\big)\cap\mathcal{N}\big(P^{T}(I-AM)A^{\frac{1}{2}}\big).
\end{displaymath}

``$\|E_{\rm TG}\|_{A}<1\Rightarrow\eqref{equiv-cond}$'': If~\eqref{equiv-cond} does not hold, then there exists a vector $\mathbf{x}$ in $\big(\mathcal{N}\big(A^{\frac{1}{2}}\overline{M}A^{\frac{1}{2}}\big)\cap\mathcal{N}\big(P^{T}(I-AM)A^{\frac{1}{2}}\big)\big)\backslash\mathcal{N}(A)$. Since $\mathbb{R}^{n}=\mathcal{R}(A)\oplus\mathcal{N}(A)$, $\mathbf{x}$ can be decomposed into $\mathbf{x}=\mathbf{x}_{1}+\mathbf{x}_{2}$, where $\mathbf{x}_{1}\in\mathcal{R}(A)\backslash\{0\}$ and $\mathbf{x}_{2}\in\mathcal{N}(A)$. Due to
\begin{displaymath}
0=P^{T}(I-AM)A^{\frac{1}{2}}\mathbf{x}=P^{T}(I-AM)A^{\frac{1}{2}}\mathbf{x}_{1}=P^{T}A^{\frac{1}{2}}\big(I-A^{\frac{1}{2}}MA^{\frac{1}{2}}\big)\mathbf{x}_{1},
\end{displaymath}
it follows that
\begin{displaymath}
\Pi\big(I-A^{\frac{1}{2}}MA^{\frac{1}{2}}\big)\mathbf{x}_{1}=0.
\end{displaymath}
In addition,
\begin{displaymath}
A^{\frac{1}{2}}\overline{M}A^{\frac{1}{2}}\mathbf{x}_{1}=A^{\frac{1}{2}}\overline{M}A^{\frac{1}{2}}(\mathbf{x}_{1}+\mathbf{x}_{2})=A^{\frac{1}{2}}\overline{M}A^{\frac{1}{2}}\mathbf{x}=0.
\end{displaymath}
Hence,
\begin{displaymath}
\mathbf{x}_{1}^{T}F_{\rm TG}\mathbf{x}_{1}=\mathbf{x}_{1}^{T}A^{\frac{1}{2}}\overline{M}A^{\frac{1}{2}}\mathbf{x}_{1}+\mathbf{x}_{1}^{T}\big(I-A^{\frac{1}{2}}M^{T}A^{\frac{1}{2}}\big)\Pi\big(I-A^{\frac{1}{2}}MA^{\frac{1}{2}}\big)\mathbf{x}_{1}=0,
\end{displaymath}
which contradicts with the condition $\|E_{\rm TG}\|_{A}<1$ (or, equivalently, $\mathbf{v}^{T}F_{\rm TG}\mathbf{v}>0$ for all $\mathbf{v}\in\mathcal{R}(A)\backslash\{0\}$).

``$\eqref{equiv-cond}\Rightarrow\|E_{\rm TG}\|_{A}<1$'': Conversely, if~\eqref{equiv-cond} holds, then $\|E_{\rm TG}\|_{A}<1$. Otherwise, there must exist a vector $\mathbf{y}\in\mathcal{R}(A)\backslash\{0\}$ such that $\mathbf{y}^{T}F_{\rm TG}\mathbf{y}=0$. Then
\begin{displaymath}
\mathbf{y}^{T}A^{\frac{1}{2}}\overline{M}A^{\frac{1}{2}}\mathbf{y}=0 \quad \text{and} \quad \mathbf{y}^{T}\big(I-A^{\frac{1}{2}}M^{T}A^{\frac{1}{2}}\big)\Pi\big(I-A^{\frac{1}{2}}MA^{\frac{1}{2}}\big)\mathbf{y}=0,
\end{displaymath}
which lead to
\begin{displaymath}
\mathbf{y}\in\mathcal{N}\big(\big(A^{\frac{1}{2}}\overline{M}A^{\frac{1}{2}}\big)^{\frac{1}{2}}\big)=\mathcal{N}\big(A^{\frac{1}{2}}\overline{M}A^{\frac{1}{2}}\big)
\end{displaymath}
and
\begin{displaymath}
\big(I-A^{\frac{1}{2}}MA^{\frac{1}{2}}\big)\mathbf{y}\in\mathcal{N}(\Pi)=\mathcal{N}\big(P^{T}A^{\frac{1}{2}}\big).
\end{displaymath}
Thus,
\begin{displaymath}
\mathbf{y}\in\mathcal{N}\big(A^{\frac{1}{2}}\overline{M}A^{\frac{1}{2}}\big)\cap\mathcal{N}\big(P^{T}(I-AM)A^{\frac{1}{2}}\big),
\end{displaymath}
which contradicts with~\eqref{equiv-cond}, because $\mathbf{y}\notin\mathcal{N}(A)$. This completes the proof.
\end{proof}

\begin{corollary}\label{cor:suff-cond}
Let $\overline{M}$ be defined by~\eqref{barM}. If $\overline{M}\succeq 0$ and
\begin{equation}\label{suff-cond}
\mathcal{N}(\overline{M})\cap\mathcal{R}(A)=\{0\},
\end{equation}
then $\|E_{\rm TG}\|_{A}<1$.
\end{corollary}

\begin{proof}
For any $\mathbf{v}\in\mathcal{N}\big(\overline{M}A^{\frac{1}{2}}\big)$, we have
\begin{displaymath}
A^{\frac{1}{2}}\mathbf{v}\in\mathcal{N}(\overline{M})\cap\mathcal{R}(A).
\end{displaymath}
It then follows from~\eqref{suff-cond} that $\mathbf{v}\in\mathcal{N}(A)$. Since $\mathbf{v}$ is arbitrary and 
\begin{displaymath}
\mathcal{N}(A)\subseteq\mathcal{N}\big(\overline{M}A^{\frac{1}{2}}\big),
\end{displaymath}
we obtain
\begin{displaymath}
\mathcal{N}\big(\overline{M}A^{\frac{1}{2}}\big)=\mathcal{N}(A),
\end{displaymath}
which, together with the positive semidefiniteness of $\overline{M}$, yields
\begin{displaymath}
\mathcal{N}\big(A^{\frac{1}{2}}\overline{M}A^{\frac{1}{2}}\big)=\mathcal{N}\big(\overline{M}^{\frac{1}{2}}A^{\frac{1}{2}}\big)=\mathcal{N}\big(\overline{M}A^{\frac{1}{2}}\big)=\mathcal{N}(A).
\end{displaymath}
The desired result then follows from Lemma~\ref{lem:equiv-cond}.
\end{proof}

\begin{remark}\label{rek:SPD-M}
Let
\begin{displaymath}
A=D+L+L^{T},
\end{displaymath}
where $D$ and $L$ denote the diagonal and strictly lower triangular parts of $A$, respectively. Without loss of generality, we assume that $D\succ 0$. If some diagonal entry of $D$ is zero, the corresponding row and column in $A$ are both zero, in which case one may solve a reduced version of~\eqref{system} obtained by deleting the zero row and column. From Corollary~\ref{cor:suff-cond}, we deduce that $\|E_{\rm TG}\|_{A}<1$ if $\overline{M}\succ 0$, which can be satisfied by the most commonly used weighted Jacobi and Gauss--Seidel smoothers.
\begin{itemize}[leftmargin=0.85cm]

\item If $M=\omega D^{-1}$ with $0<\omega<\frac{2}{\lambda_{\max}(D^{-1}A)}$, then
\begin{displaymath}
\overline{M}=\omega^{2}D^{-1}\big(2\omega^{-1}D-A\big)D^{-1}\succ 0.
\end{displaymath}

\item If $M=(D+L)^{-1}$, then
\begin{displaymath}
\overline{M}=(D+L)^{-T}D(D+L)^{-1}\succ 0.
\end{displaymath}

\end{itemize}
In both cases, it holds that $\|E_{\rm TG}\|_{A}<1$.
\end{remark}

\section{Convergence analysis of Algorithm~\ref{alg:TG}} \label{sec:exact}

In this section, we present a succinct identity for characterizing the convergence factor $\|E_{\rm TG}\|_{A}$.

Let
\begin{displaymath}
r=\rank(A) \quad \text{and} \quad s=\rank(A_{\rm c}).
\end{displaymath}
Since $A_{\rm c}=P^{T}AP$, it follows that $s\leq r$. In the extreme case $s=r$, we have
\begin{displaymath}
\rank(\Pi)=\rank\big(A^{\frac{1}{2}}P\big)=s=r,
\end{displaymath}
which, combined with the fact $\mathcal{R}(\Pi)\subseteq\mathcal{R}(A)$, leads to
\begin{displaymath}
\mathcal{R}(\Pi)=\mathcal{R}(A).
\end{displaymath}
Then
\begin{displaymath}
(I-\Pi)\mathbf{v}=0 \quad \forall\,\mathbf{v}\in\mathcal{R}(A),
\end{displaymath}
which, together with~\eqref{ETG-Anorm}, yields $\|E_{\rm TG}\|_{A}=0$. However, such a case seldom occurs in practice. In what follows, we only consider the nontrivial case $s<r$.

We first prove a technical lemma, which gives a reformulation of the convergence factor $\|E_{\rm TG}\|_{A}$.

\begin{lemma}\label{lem:refor}
The convergence factor $\|E_{\rm TG}\|_{A}$ can be reformulated as
\begin{equation}\label{identity0}
\|E_{\rm TG}\|_{A}=\sqrt{1-\lambda_{n-r+1}(F_{\rm TG})},
\end{equation}
where $r=\rank(A)$ and $F_{\rm TG}$ is given by~\eqref{F-TG}.
\end{lemma}

\begin{proof}
It can be easily seen from~\eqref{pi} and~\eqref{F-TG} that
\begin{equation}\label{nullAinF}
\mathcal{N}(A)\subseteq\mathcal{N}(F_{\rm TG}).
\end{equation}
For any $\mathbf{v}\in\mathcal{N}(F_{\rm TG})$, we have
\begin{displaymath}
0=\mathbf{v}^{T}F_{\rm TG}\mathbf{v}=\mathbf{v}^{T}A^{\frac{1}{2}}\overline{M}A^{\frac{1}{2}}\mathbf{v}+\mathbf{v}^{T}\big(I-A^{\frac{1}{2}}M^{T}A^{\frac{1}{2}}\big)\Pi\big(I-A^{\frac{1}{2}}MA^{\frac{1}{2}}\big)\mathbf{v},
\end{displaymath}
from which we deduce that
\begin{displaymath}
\mathbf{v}\in\mathcal{N}\big(A^{\frac{1}{2}}\overline{M}A^{\frac{1}{2}}\big) \quad \text{and} \quad \big(I-A^{\frac{1}{2}}MA^{\frac{1}{2}}\big)\mathbf{v}\in\mathcal{N}(\Pi)=\mathcal{N}\big(P^{T}A^{\frac{1}{2}}\big).
\end{displaymath}
Then
\begin{displaymath}
\mathbf{v}\in\mathcal{N}\big(A^{\frac{1}{2}}\overline{M}A^{\frac{1}{2}}\big)\cap\mathcal{N}\big(P^{T}(I-AM)A^{\frac{1}{2}}\big)=\mathcal{N}(A),
\end{displaymath}
where we have used the condition~\eqref{equiv-cond}. The arbitrariness of $\mathbf{v}$ implies that
\begin{equation}\label{nullFinA}
\mathcal{N}(F_{\rm TG})\subseteq\mathcal{N}(A).
\end{equation}
Combining~\eqref{nullAinF} and~\eqref{nullFinA}, we obtain
\begin{displaymath}
\mathcal{N}(F_{\rm TG})=\mathcal{N}(A)
\end{displaymath}
and hence
\begin{displaymath}
\mathcal{R}(F_{\rm TG})=\mathcal{R}(A).
\end{displaymath}

Since $F_{\rm TG}\succeq 0$ and $\rank(F_{\rm TG})=r$, there exists an orthogonal matrix $Q_{1}\in\mathbb{R}^{n\times n}$ such that
\begin{displaymath}
Q_{1}^{T}F_{\rm TG}Q_{1}=\begin{pmatrix}
\Lambda_{r} & 0 \\
0 & 0
\end{pmatrix},
\end{displaymath}
where $\Lambda_{r}\in\mathbb{R}^{r\times r}$ is diagonal and $\Lambda_{r}\succ 0$. We then have
\begin{displaymath}
\mathcal{R}(A)=\mathcal{R}(F_{\rm TG})=\mathcal{R}\bigg(Q_{1}\begin{pmatrix}I_{r} \\ 0\end{pmatrix}\bigg).
\end{displaymath}
Thus,
\begin{align*}
\|E_{\rm TG}\|_{A}^{2}&=1-\min_{\mathbf{v}\in\mathcal{R}(A)\backslash\{0\}}\frac{\mathbf{v}^{T}F_{\rm TG}\mathbf{v}}{\mathbf{v}^{T}\mathbf{v}}\\
&=1-\min_{\mathbf{v}_{r}\in\mathbb{R}^{r}\backslash\{0\}}\frac{\mathbf{v}_{r}^{T}\big(I_{r} \,\ 0\big)Q_{1}^{T}F_{\rm TG}Q_{1}\begin{pmatrix}I_{r} \\ 0\end{pmatrix}\mathbf{v}_{r}}{\mathbf{v}_{r}^{T}\big(I_{r} \,\ 0\big)Q_{1}^{T}Q_{1}\begin{pmatrix}I_{r} \\ 0\end{pmatrix}\mathbf{v}_{r}}\\
&=1-\min_{\mathbf{v}_{r}\in\mathbb{R}^{r}\backslash\{0\}}\frac{\mathbf{v}_{r}^{T}\Lambda_{r}\mathbf{v}_{r}}{\mathbf{v}_{r}^{T}\mathbf{v}_{r}}\\
&=1-\lambda_{\min}(\Lambda_{r})\\
&=1-\lambda_{n-r+1}(F_{\rm TG}),
\end{align*}
which yields the expression~\eqref{identity0}.
\end{proof}

\begin{remark}
A symmetric two-grid scheme can be obtained by adding the postsmoothing step
\begin{displaymath}
\mathbf{u}_{\rm STG}\gets\mathbf{u}_{\rm TG}+M^{T}(\mathbf{f}-A\mathbf{u}_{\rm TG})
\end{displaymath}
after the prolongation step of Algorithm~\ref{alg:TG}. The error propagation of the symmetric two-grid scheme can be expressed as
\begin{displaymath}
\mathbf{u}-\mathbf{u}_{\rm STG}=E_{\rm STG}\big(\mathbf{u}-\mathbf{u}^{(0)}\big)
\end{displaymath}
with
\begin{displaymath}
E_{\rm STG}=(I-M^{T}A)\big(I-PA_{\rm c}^{\dagger}P^{T}A\big)(I-MA).
\end{displaymath}
Since
\begin{displaymath}
A^{\frac{1}{2}}E_{\rm STG}=\big(I-A^{\frac{1}{2}}M^{T}A^{\frac{1}{2}}\big)(I-\Pi)\big(I-A^{\frac{1}{2}}MA^{\frac{1}{2}}\big)A^{\frac{1}{2}}=(I-F_{\rm TG})A^{\frac{1}{2}},
\end{displaymath}
we obtain
\begin{displaymath}
\|E_{\rm STG}\|_{A}=\max_{\mathbf{v}\in\mathcal{R}(A)\backslash\{0\}}\frac{\|(I-F_{\rm TG})\mathbf{v}\|_{2}}{\|\mathbf{v}\|_{2}}.
\end{displaymath}
Then
\begin{align*}
\|E_{\rm STG}\|_{A}^{2}&=\max_{\mathbf{v}\in\mathcal{R}(A)\backslash\{0\}}\frac{\mathbf{v}^{T}(I-F_{\rm TG})^{2}\mathbf{v}}{\mathbf{v}^{T}\mathbf{v}}\\
&=\max_{\mathbf{v}_{r}\in\mathbb{R}^{r}\backslash\{0\}}\frac{\mathbf{v}_{r}^{T}\big(I_{r} \,\ 0\big)Q_{1}^{T}(I-F_{\rm TG})^{2}Q_{1}\begin{pmatrix}I_{r} \\ 0\end{pmatrix}\mathbf{v}_{r}}{\mathbf{v}_{r}^{T}\big(I_{r} \,\ 0\big)Q_{1}^{T}Q_{1}\begin{pmatrix}I_{r} \\ 0\end{pmatrix}\mathbf{v}_{r}}\\
&=\max_{\mathbf{v}_{r}\in\mathbb{R}^{r}\backslash\{0\}}\frac{\mathbf{v}_{r}^{T}(I_{r}-\Lambda_{r})^{2}\mathbf{v}_{r}}{\mathbf{v}_{r}^{T}\mathbf{v}_{r}}.
\end{align*}
Due to
\begin{displaymath}
Q_{1}\begin{pmatrix}
I_{r}-\Lambda_{r} & 0 \\
0 & I_{n-r}
\end{pmatrix}Q_{1}^{T}=I-F_{\rm TG}\succeq 0 \quad \text{and} \quad \Lambda_{r}\succ 0,
\end{displaymath}
it follows that $\lambda(\Lambda_{r})\subset(0,1]$. Hence,
\begin{displaymath}
\|E_{\rm STG}\|_{A}^{2}=\lambda_{\max}\big((I_{r}-\Lambda_{r})^{2}\big)=\big(1-\lambda_{\min}(\Lambda_{r})\big)^{2}=\|E_{\rm TG}\|_{A}^{4},
\end{displaymath}
which yields
\begin{displaymath}
\|E_{\rm STG}\|_{A}=\|E_{\rm TG}\|_{A}^{2}.
\end{displaymath}
This shows that there is no essential difference between the convergence analysis of Algorithm~\ref{alg:TG} and that of its symmetrized version.
\end{remark}

Based on~\eqref{identity0} and its proof, we can further derive the following identity, which is the main result of this section.

\begin{theorem}\label{thm:ide}
Define
\begin{equation}\label{tildM}
\widetilde{M}:=M+M^{T}-MAM^{T}
\end{equation}
and
\begin{equation}\label{piA}
\Pi_{A}:=PA_{\rm c}^{\dagger}P^{T}A.
\end{equation}
Then, it holds that
\begin{equation}\label{identity}
\|E_{\rm TG}\|_{A}=\sqrt{1-\lambda_{n-r+s+1}\big(\widetilde{M}A(I-\Pi_{A})\big)},
\end{equation}
where $r=\rank(A)$ and $s=\rank(A_{\rm c})$.
\end{theorem}

\begin{proof}
Due to
\begin{displaymath}
F_{\rm TG}=I-\big(I-A^{\frac{1}{2}}M^{T}A^{\frac{1}{2}}\big)(I-\Pi)\big(I-A^{\frac{1}{2}}MA^{\frac{1}{2}}\big),
\end{displaymath}
it follows that
\begin{align*}
\lambda(F_{\rm TG})&=\lambda\big(I-\big(I-A^{\frac{1}{2}}MA^{\frac{1}{2}}\big)\big(I-A^{\frac{1}{2}}M^{T}A^{\frac{1}{2}}\big)(I-\Pi)\big)\\
&=\lambda\big(I-\big(I-A^{\frac{1}{2}}\widetilde{M}A^{\frac{1}{2}}\big)(I-\Pi)\big)\\
&=\lambda\big(A^{\frac{1}{2}}\widetilde{M}A^{\frac{1}{2}}(I-\Pi)+\Pi\big).
\end{align*}
Since $\Pi$ is an $L^{2}$-orthogonal projector and $\rank(\Pi)=s$, there exists an orthogonal matrix $Q_{2}\in\mathbb{R}^{n\times n}$ such that
\begin{equation}\label{Q2pi}
Q_{2}^{T}\Pi Q_{2}=\begin{pmatrix}
I_{s} & 0 \\
0 & 0
\end{pmatrix}.
\end{equation}
Let
\begin{equation}\label{Q2AtM}
Q_{2}^{T}A^{\frac{1}{2}}\widetilde{M}A^{\frac{1}{2}}Q_{2}=\begin{pmatrix}
X_{1} & X_{2} \\
X_{2}^{T} & X_{3}
\end{pmatrix},
\end{equation}
where $X_{1}\in\mathbb{R}^{s\times s}$, $X_{2}\in\mathbb{R}^{s\times(n-s)}$, and $X_{3}\in\mathbb{R}^{(n-s)\times(n-s)}$. Then
\begin{displaymath}
A^{\frac{1}{2}}\widetilde{M}A^{\frac{1}{2}}(I-\Pi)+\Pi=Q_{2}\begin{pmatrix}
I_{s} & X_{2} \\
0 & X_{3}
\end{pmatrix}Q_{2}^{T}.
\end{displaymath}
Note that $A^{\frac{1}{2}}\widetilde{M}A^{\frac{1}{2}}(I-\Pi)+\Pi$ has $r$ positive eigenvalues and $n-r$ zero eigenvalues; see the proof of~\eqref{identity0}. We conclude that $X_{3}$ has $r-s$ positive eigenvalues and $n-r$ zero eigenvalues.

Direct computation yields
\begin{displaymath}
\big(I-A^{\frac{1}{2}}MA^{\frac{1}{2}}\big)\big(I-A^{\frac{1}{2}}M^{T}A^{\frac{1}{2}}\big)=I-A^{\frac{1}{2}}\widetilde{M}A^{\frac{1}{2}}=Q_{2}\begin{pmatrix}
I_{s}-X_{1} & -X_{2} \\
-X_{2}^{T} & I_{n-s}-X_{3}
\end{pmatrix}Q_{2}^{T}.
\end{displaymath}
The positive semidefiniteness of $I-A^{\frac{1}{2}}\widetilde{M}A^{\frac{1}{2}}$ implies that $\lambda_{\max}(X_{3})\leq 1$. Thus,
\begin{displaymath}
\lambda_{n-r+1}(F_{\rm TG})=\lambda_{n-r+1}\big(A^{\frac{1}{2}}\widetilde{M}A^{\frac{1}{2}}(I-\Pi)+\Pi\big)=\lambda_{n-r+1}(X_{3}).
\end{displaymath}
Since
\begin{displaymath}
(I-\Pi)A^{\frac{1}{2}}\widetilde{M}A^{\frac{1}{2}}(I-\Pi)=Q_{2}\begin{pmatrix}
0 & 0 \\
0 & X_{3}
\end{pmatrix}Q_{2}^{T},
\end{displaymath}
we obtain
\begin{align*}
\lambda_{n-r+1}(X_{3})&=\lambda_{n-r+s+1}\big((I-\Pi)A^{\frac{1}{2}}\widetilde{M}A^{\frac{1}{2}}(I-\Pi)\big)\\
&=\lambda_{n-r+s+1}\big(\widetilde{M}A^{\frac{1}{2}}(I-\Pi)A^{\frac{1}{2}}\big)\\
&=\lambda_{n-r+s+1}\big(\widetilde{M}A(I-\Pi_{A})\big).
\end{align*}
The identity~\eqref{identity} then follows by using~\eqref{identity0}.
\end{proof}

\begin{remark}
In view of~\eqref{tildM}, we have
\begin{displaymath}
I-A^{\frac{1}{2}}\widetilde{M}A^{\frac{1}{2}}=\big(I-A^{\frac{1}{2}}MA^{\frac{1}{2}}\big)\big(I-A^{\frac{1}{2}}M^{T}A^{\frac{1}{2}}\big)\succeq 0,
\end{displaymath}
which leads to
\begin{displaymath}
\lambda_{\max}\big(A^{\frac{1}{2}}\widetilde{M}A^{\frac{1}{2}}\big)\leq 1.
\end{displaymath}
In addition,
\begin{align*}
\lambda\big(A^{\frac{1}{2}}\widetilde{M}A^{\frac{1}{2}}\big)&=\lambda\big(I-\big(I-A^{\frac{1}{2}}MA^{\frac{1}{2}}\big)\big(I-A^{\frac{1}{2}}M^{T}A^{\frac{1}{2}}\big)\big)\\
&=\lambda\big(I-\big(I-A^{\frac{1}{2}}M^{T}A^{\frac{1}{2}}\big)\big(I-A^{\frac{1}{2}}MA^{\frac{1}{2}}\big)\big)\\
&=\lambda\big(A^{\frac{1}{2}}\overline{M}A^{\frac{1}{2}}\big)\subset[0,+\infty),
\end{align*}
where $\overline{M}$ is defined by~\eqref{barM}. Hence, $A^{\frac{1}{2}}\widetilde{M}A^{\frac{1}{2}}$ is an SPSD matrix with eigenvalues contained in $[0,1]$. Since
\begin{displaymath}
A^{\frac{1}{2}}\widetilde{M}A^{\frac{1}{2}}-\big(A^{\frac{1}{2}}\widetilde{M}A^{\frac{1}{2}}\big)^{\frac{1}{2}}(I-\Pi)\big(A^{\frac{1}{2}}\widetilde{M}A^{\frac{1}{2}}\big)^{\frac{1}{2}}\succeq 0,
\end{displaymath}
it follows that
\begin{displaymath}
\lambda_{n-r+s+1}\big(\widetilde{M}A(I-\Pi_{A})\big)=\lambda_{n-r+s+1}\big(A^{\frac{1}{2}}\widetilde{M}A^{\frac{1}{2}}(I-\Pi)\big)\leq\lambda_{n-r+s+1}(\widetilde{M}A),
\end{displaymath}
which, combined with~\eqref{identity}, gives
\begin{equation}\label{ETG-low}
\|E_{\rm TG}\|_{A}\geq\sqrt{1-\lambda_{n-r+s+1}(\widetilde{M}A)}.
\end{equation}
In the special case where $A\succ 0$ and $\rank(P)=n_{\rm c}$ (i.e., $r=n$ and $s=n_{\rm c}$), \eqref{ETG-low} reduces to
\begin{equation}\label{ETG-low1}
\|E_{\rm TG}\|_{A}\geq\sqrt{1-\lambda_{n_{\rm c}+1}(\widetilde{M}A)}.
\end{equation}
As proved in~\cite[Lemma~1]{Brannick2018}, the lower bound in~\eqref{ETG-low1} is attainable. Furthermore, \eqref{ETG-low1} and its sharpness still hold when $A\succeq 0$, $\mathcal{N}(A)\subset\mathcal{R}(P)$, and $\rank(P)=n_{\rm c}$; see~\cite[Theorem~5.6]{XZ2017}.
\end{remark}

Using~\eqref{identity}, we can prove the following corollary.

\begin{corollary}
Let
\begin{displaymath}
d_{1}=\dim\big(\mathcal{R}(P)\cap\mathcal{N}(A)\big) \quad \text{and} \quad d_{2}=\dim\big(\mathcal{R}(P)+\mathcal{N}(A)\big).
\end{displaymath}
Then, it holds that
\begin{subequations}
\begin{align}
\|E_{\rm TG}\|_{A}&=\sqrt{1-\lambda_{n-r+\rank(P)-d_{1}+1}\big(\widetilde{M}A(I-\Pi_{A})\big)}\label{identity-d1}\\
&=\sqrt{1-\lambda_{d_{2}+1}\big(\widetilde{M}A(I-\Pi_{A})\big)}.\label{identity-d2}
\end{align}
\end{subequations}
\end{corollary}

\begin{proof}
Assume that $\{\mathbf{p}_{1},\ldots,\mathbf{p}_{d_{1}}\}$ is a basis for $\mathcal{R}(P)\cap\mathcal{N}(A)$. Such a set can be expanded as a basis for $\mathcal{R}(P)$, which is denoted by $\big\{\mathbf{p}_{1},\ldots,\mathbf{p}_{d_{1}},\mathbf{p}_{d_{1}+1},\ldots,\mathbf{p}_{\rank(P)}\big\}$. Let
\begin{displaymath}
P=\big(\mathbf{p}_{1},\ldots,\mathbf{p}_{\rank(P)}\big)W,
\end{displaymath}
where $W\in\mathbb{R}^{\rank(P)\times n_{\rm c}}$ is of full row rank. Then
\begin{displaymath}
AP=\big(0,\ldots,0,A\mathbf{p}_{d_{1}+1},\ldots,A\mathbf{p}_{\rank(P)}\big)W.
\end{displaymath}
On the one hand,
\begin{displaymath}
\rank(AP)\leq\rank\big(\big(A\mathbf{p}_{d_{1}+1},\ldots,A\mathbf{p}_{\rank(P)}\big)\big).
\end{displaymath}
On the other hand,
\begin{align*}
\rank(AP)&\geq\rank\big(\big(A\mathbf{p}_{d_{1}+1},\ldots,A\mathbf{p}_{\rank(P)}\big)\big)+\rank(W)-\rank(P)\\
&=\rank\big(\big(A\mathbf{p}_{d_{1}+1},\ldots,A\mathbf{p}_{\rank(P)}\big)\big).
\end{align*}
Note that the set $\big\{A\mathbf{p}_{d_{1}+1},\ldots,A\mathbf{p}_{\rank(P)}\big\}$ is linearly independent. We then have
\begin{displaymath}
\rank(AP)=\rank\big(\big(A\mathbf{p}_{d_{1}+1},\ldots,A\mathbf{p}_{\rank(P)}\big)\big)=\rank(P)-d_{1}.
\end{displaymath}
Due to
\begin{displaymath}
\mathcal{N}(A_{\rm c})=\mathcal{N}(A^{\frac{1}{2}}P)=\mathcal{N}(AP),
\end{displaymath}
it follows that
\begin{displaymath}
s=\rank(A_{\rm c})=\rank(AP)=\rank(P)-d_{1},
\end{displaymath}
which, together with~\eqref{identity}, yields the identity~\eqref{identity-d1}.

The identity~\eqref{identity-d2} follows from~\eqref{identity-d1} and the fact
\begin{displaymath}
d_{1}=\rank(P)+n-r-d_{2},
\end{displaymath}
because
\begin{displaymath}
\dim\big(\mathcal{R}(P)\big)+\dim\big(\mathcal{N}(A)\big)=\dim\big(\mathcal{R}(P)+\mathcal{N}(A)\big)+\dim\big(\mathcal{R}(P)\cap\mathcal{N}(A)\big).
\end{displaymath}
This completes the proof.
\end{proof}

In particular, if $\mathcal{N}(A)\subset\mathcal{R}(P)$, then we have the following result, which provides an alternative to the convergence identity proved in~\cite[Theorem~5.3]{XZ2017}.

\begin{corollary}
If $\mathcal{N}(A)\subset\mathcal{R}(P)$, then
\begin{displaymath}
\|E_{\rm TG}\|_{A}=\sqrt{1-\lambda_{\rank(P)+1}\big(\widetilde{M}A(I-\Pi_{A})\big)}.
\end{displaymath}	
\end{corollary}

\begin{proof}
If $\mathcal{N}(A)\subset\mathcal{R}(P)$, then
\begin{displaymath}
d_{1}=n-r \quad \text{and} \quad d_{2}=\rank(P).
\end{displaymath}
The desired result then follows immediately from~\eqref{identity-d1} or~\eqref{identity-d2}.
\end{proof}

\section{An inexact variant of Algorithm~\ref{alg:TG} and its analysis} \label{sec:inexact}

In this section, we are devoted to analyzing the convergence of an inexact variant of Algorithm~\ref{alg:TG} (that is, the coarse-grid system~\eqref{coarse-system} is solved approximately), which is described by Algorithm~\ref{alg:iTG} below. The output of its correction step, $\hat{\mathbf{e}}_{\rm c}=\mathscr{B}_{\rm c}\llbracket\mathbf{r}_{\rm c}\rrbracket$, is expected to be a good approximation to $\mathbf{e}_{\rm c}=A_{\rm c}^{\dagger}\mathbf{r}_{\rm c}$. Here, $\mathscr{B}_{\rm c}\llbracket\cdot\rrbracket\approx A_{\rm c}^{\dagger}$ is a general mapping from $\mathbb{R}^{n_{\rm c}}$ to $\mathbb{R}^{n_{\rm c}}$.

\begin{algorithm}[!htbp]

\caption{\ Inexact two-grid method.}\label{alg:iTG}

\smallskip

\begin{algorithmic}[1]

\State Smoothing: $\mathbf{u}^{(1)}\gets\mathbf{u}^{(0)}+M\big(\mathbf{f}-A\mathbf{u}^{(0)}\big)$

\smallskip

\State Restriction: $\mathbf{r}_{\rm c}\gets P^{T}\big(\mathbf{f}-A\mathbf{u}^{(1)}\big)$

\smallskip

\State Coarse-grid correction: $\hat{\mathbf{e}}_{\rm c}\gets\mathscr{B}_{\rm c}\llbracket\mathbf{r}_{\rm c}\rrbracket$

\smallskip

\State Prolongation: $\mathbf{u}_{\rm ITG}\gets\mathbf{u}^{(1)}+P\hat{\mathbf{e}}_{\rm c}$

\smallskip

\end{algorithmic}

\end{algorithm}

This section will be divided into two parts: the first part focuses on a theoretical estimate of $\|\mathbf{u}-\mathbf{u}_{\rm ITG}\|_{A}$; the second part provides a numerical example to illustrate our estimate.

\subsection{Convergence analysis}

To analyze the convergence of Algorithm~\ref{alg:iTG}, we need the following technical lemma.

\begin{lemma}
Let $\Pi$, $\widetilde{M}$, and $\Pi_{A}$ be defined by~\eqref{pi}, \eqref{tildM}, and~\eqref{piA}, respectively. Then
\begin{equation}\label{maxeig-ide}
\lambda_{\max}\big(\big(I-A^{\frac{1}{2}}M^{T}A^{\frac{1}{2}}\big)\Pi\big(I-A^{\frac{1}{2}}MA^{\frac{1}{2}}\big)\big)=1-\delta_{\rm TG},
\end{equation}
where
\begin{equation}\label{delta-TG}
\delta_{\rm TG}=\begin{cases}
\lambda_{n-s+1}(\widetilde{M}A\Pi_{A}) & \text{if $\mathcal{N}\big(A^{\frac{1}{2}}\widetilde{M}A^{\frac{1}{2}}\big)\cap\mathcal{R}\big(A^{\frac{1}{2}}P\big)=\{0\}$},\\[2pt]
0 & \text{otherwise}.
\end{cases}
\end{equation}
\end{lemma}

\begin{proof}
Since $A^{\frac{1}{2}}\widetilde{M}A^{\frac{1}{2}}\succeq 0$ and $\lambda\big(A^{\frac{1}{2}}\widetilde{M}A^{\frac{1}{2}}\big)\subset[0,1]$, we get from~\eqref{Q2AtM} that $X_{1}\succeq 0$ and $\lambda(X_{1})\subset[0,1]$. By~\eqref{Q2pi} and~\eqref{Q2AtM}, we have
\begin{displaymath}
\big(I-A^{\frac{1}{2}}\widetilde{M}A^{\frac{1}{2}}\big)\Pi=Q_{2}\begin{pmatrix}
I_{s}-X_{1} & 0 \\
-X_{2}^{T} & 0
\end{pmatrix}Q_{2}^{T}.
\end{displaymath}
Then
\begin{align}
\lambda_{\max}\big(\big(I-A^{\frac{1}{2}}M^{T}A^{\frac{1}{2}}\big)\Pi\big(I-A^{\frac{1}{2}}MA^{\frac{1}{2}}\big)\big)&=\lambda_{\max}\big(\big(I-A^{\frac{1}{2}}\widetilde{M}A^{\frac{1}{2}}\big)\Pi\big)\label{maxeig-ide0}\\
&=1-\lambda_{\min}(X_{1}).\notag
\end{align}

Direct computation yields
\begin{displaymath}
\Pi A^{\frac{1}{2}}\widetilde{M}A^{\frac{1}{2}}\Pi=Q_{2}\begin{pmatrix}
X_{1} & 0 \\
0 & 0
\end{pmatrix}Q_{2}^{T},
\end{displaymath}
where we have used the relations~\eqref{Q2pi} and~\eqref{Q2AtM}. It follows that
\begin{equation}\label{rank-X1}
\rank(X_{1})=\rank\big(\Pi A^{\frac{1}{2}}\widetilde{M}A^{\frac{1}{2}}\Pi\big).
\end{equation}
Next, we prove that $\rank(X_{1})=s$ (i.e., $X_{1}\succ 0$) if and only if
\begin{equation}\label{rank-rc}
\mathcal{N}\big(A^{\frac{1}{2}}\widetilde{M}A^{\frac{1}{2}}\big)\cap\mathcal{R}\big(A^{\frac{1}{2}}P\big)=\{0\}.
\end{equation}
\begin{itemize}[leftmargin=0.85cm]

\item Recall that $\mathcal{R}(\Pi)=\mathcal{R}\big(A^{\frac{1}{2}}P\big)$ and $\mathcal{N}(\Pi)=\mathcal{N}\big(P^{T}A^{\frac{1}{2}}\big)$. If~\eqref{rank-rc} holds, then
\begin{displaymath}
\mathcal{N}\big(\Pi A^{\frac{1}{2}}\widetilde{M}A^{\frac{1}{2}}\Pi\big)\subseteq\mathcal{N}(\Pi),
\end{displaymath}
which, combined with the fact $\mathcal{N}(\Pi)\subseteq\mathcal{N}\big(\Pi A^{\frac{1}{2}}\widetilde{M}A^{\frac{1}{2}}\Pi\big)$, leads to
\begin{displaymath}
\mathcal{N}\big(\Pi A^{\frac{1}{2}}\widetilde{M}A^{\frac{1}{2}}\Pi\big)=\mathcal{N}(\Pi)=\mathcal{N}\big(P^{T}A^{\frac{1}{2}}\big).
\end{displaymath}
Hence,
\begin{displaymath}
\rank\big(\Pi A^{\frac{1}{2}}\widetilde{M}A^{\frac{1}{2}}\Pi\big)=\rank\big(P^{T}A^{\frac{1}{2}}\big)=\rank(A_{\rm c})=s,
\end{displaymath}
which, together with~\eqref{rank-X1}, yields $\rank(X_{1})=s$.

\item Conversely, if $\rank(X_{1})=s$, then~\eqref{rank-rc} holds. Otherwise, there must exist a vector $\mathbf{v}\in\mathbb{R}^{n}\backslash\mathcal{N}(\Pi)$ such that
\begin{displaymath}
A^{\frac{1}{2}}\widetilde{M}A^{\frac{1}{2}}\Pi\mathbf{v}=0,
\end{displaymath}
i.e., $\mathbf{v}\in\mathcal{N}\big(\Pi A^{\frac{1}{2}}\widetilde{M}A^{\frac{1}{2}}\Pi\big)\backslash\mathcal{N}(\Pi)$. This contradicts with the fact
\begin{displaymath}
\mathcal{N}\big(\Pi A^{\frac{1}{2}}\widetilde{M}A^{\frac{1}{2}}\Pi\big)=\mathcal{N}(\Pi),
\end{displaymath}
which follows from the relations
\begin{displaymath}
\mathcal{N}(\Pi)\subseteq\mathcal{N}\big(\Pi A^{\frac{1}{2}}\widetilde{M}A^{\frac{1}{2}}\Pi\big)
\end{displaymath}
and
\begin{displaymath}
\rank(\Pi)=s=\rank(X_{1})=\rank\big(\Pi A^{\frac{1}{2}}\widetilde{M}A^{\frac{1}{2}}\Pi\big).
\end{displaymath}
\end{itemize}
The above analysis shows that
\begin{equation}\label{mineig-X1}
\lambda_{\min}(X_{1})=\begin{cases}
\lambda_{n-s+1}\big(\Pi A^{\frac{1}{2}}\widetilde{M}A^{\frac{1}{2}}\Pi\big) & \text{if~\eqref{rank-rc} holds},\\[2pt]
0 & \text{otherwise}.
\end{cases}
\end{equation}

Note that
\begin{equation}\label{mineig-X1-1}
\lambda\big(\Pi A^{\frac{1}{2}}\widetilde{M}A^{\frac{1}{2}}\Pi\big)=\lambda\big(\widetilde{M}A^{\frac{1}{2}}\Pi A^{\frac{1}{2}}\big)=\lambda(\widetilde{M}A\Pi_{A}).
\end{equation}
Combining~\eqref{maxeig-ide0}, \eqref{mineig-X1}, and~\eqref{mineig-X1-1}, we can arrive at the identity~\eqref{maxeig-ide}.
\end{proof}

\begin{remark}
If $\widetilde{M}\succeq 0$ and $\mathcal{N}(\widetilde{M})\cap\mathcal{R}(A)=\{0\}$, then~\eqref{rank-rc} holds, since
\begin{displaymath}
\mathcal{N}\big(A^{\frac{1}{2}}\widetilde{M}A^{\frac{1}{2}}\big)=\mathcal{N}\big(\widetilde{M}^{\frac{1}{2}}A^{\frac{1}{2}}\big)=\mathcal{N}\big(\widetilde{M}A^{\frac{1}{2}}\big)=\mathcal{N}(A)
\end{displaymath}
and
\begin{displaymath}
\mathcal{R}\big(A^{\frac{1}{2}}P\big)\subseteq\mathcal{R}(A)=\mathcal{N}(A)^{\perp}.
\end{displaymath}
Obviously, a sufficient condition for~\eqref{rank-rc} to be valid is $\widetilde{M}\succ 0$, which can be satisfied by the weighted Jacobi and Gauss--Seidel smoothers, as discussed in Remark~\ref{rek:SPD-M}.
\end{remark}

We are now ready to present a convergence estimate for Algorithm~\ref{alg:iTG}.

\begin{theorem}\label{thm:iTG}
Let $\rho,\varepsilon\in[0,1)$ be two parameters. If
\begin{equation}\label{TG-est}
\|\mathbf{u}-\mathbf{u}_{\rm TG}\|_{A}\leq\rho\|\mathbf{u}-\mathbf{u}^{(0)}\|_{A}
\end{equation}
and
\begin{equation}\label{accu-Bc}
\|\mathbf{e}_{\rm c}-\hat{\mathbf{e}}_{\rm c}\|_{A_{\rm c}}\leq\varepsilon\|\mathbf{e}_{\rm c}\|_{A_{\rm c}},
\end{equation}
then
\begin{equation}\label{iTG-est}
\|\mathbf{u}-\mathbf{u}_{\rm ITG}\|_{A}\leq\sqrt{\rho^{2}+\varepsilon^{2}\big(1-\max\big\{\rho^{2}+\lambda_{n-r+1}(\widetilde{M}A),\,\delta_{\rm TG}\big\}\big)}\|\mathbf{u}-\mathbf{u}^{(0)}\|_{A},
\end{equation}
where $\widetilde{M}$ and $\delta_{\rm TG}$ are given by~\eqref{tildM} and~\eqref{delta-TG}, respectively.
\end{theorem}

\begin{proof}
The condition~\eqref{accu-Bc} reads
\begin{equation}\label{rcB-low}
2\mathbf{r}_{\rm c}^{T}\hat{\mathbf{e}}_{\rm c}\geq(1-\varepsilon^{2})\|\mathbf{e}_{\rm c}\|_{A_{\rm c}}^{2}+\|\hat{\mathbf{e}}_{\rm c}\|_{A_{\rm c}}^{2}.
\end{equation}
From Algorithm~\ref{alg:iTG}, we have
\begin{displaymath}
\mathbf{u}-\mathbf{u}_{\rm ITG}=\mathbf{u}-\mathbf{u}^{(1)}-P\hat{\mathbf{e}}_{\rm c}.
\end{displaymath}
Then
\begin{align*}
\|\mathbf{u}-\mathbf{u}_{\rm ITG}\|_{A}^{2}&=(\mathbf{u}-\mathbf{u}^{(1)}-P\hat{\mathbf{e}}_{\rm c})^{T}A(\mathbf{u}-\mathbf{u}^{(1)}-P\hat{\mathbf{e}}_{\rm c})\\
&=\|\mathbf{u}-\mathbf{u}^{(1)}\|_{A}^{2}-2(\mathbf{u}-\mathbf{u}^{(1)})^{T}AP\hat{\mathbf{e}}_{\rm c}+\|\hat{\mathbf{e}}_{\rm c}\|_{A_{\rm c}}^{2}\\
&=\|\mathbf{u}-\mathbf{u}^{(1)}\|_{A}^{2}-2\mathbf{r}_{\rm c}^{T}\hat{\mathbf{e}}_{\rm c}+\|\hat{\mathbf{e}}_{\rm c}\|_{A_{\rm c}}^{2}\\
&\leq\|\mathbf{u}-\mathbf{u}^{(1)}\|_{A}^{2}-(1-\varepsilon^{2})\|\mathbf{e}_{\rm c}\|_{A_{\rm c}}^{2} \quad (\text{using~\eqref{rcB-low}})\\
&=(\mathbf{u}-\mathbf{u}^{(1)})^{T}A(\mathbf{u}-\mathbf{u}^{(1)})-(1-\varepsilon^{2})\mathbf{r}_{\rm c}^{T}A_{\rm c}^{\dagger}\mathbf{r}_{\rm c}\\
&=(\mathbf{u}-\mathbf{u}^{(1)})^{T}A(\mathbf{u}-\mathbf{u}^{(1)})-(1-\varepsilon^{2})(\mathbf{u}-\mathbf{u}^{(1)})^{T}A^{\frac{1}{2}}\Pi A^{\frac{1}{2}}(\mathbf{u}-\mathbf{u}^{(1)})\\
&=(\mathbf{u}-\mathbf{u}^{(1)})^{T}A^{\frac{1}{2}}\big(I-(1-\varepsilon^{2})\Pi\big)A^{\frac{1}{2}}(\mathbf{u}-\mathbf{u}^{(1)}).
\end{align*}
Due to
\begin{displaymath}
\mathbf{u}-\mathbf{u}^{(1)}=(I-MA)(\mathbf{u}-\mathbf{u}^{(0)}),
\end{displaymath}
it follows that
\begin{equation}\label{iTG-est0}
\|\mathbf{u}-\mathbf{u}_{\rm ITG}\|_{A}^{2}\leq(\mathbf{u}-\mathbf{u}^{(0)})^{T}\big(E_{1}-(1-\varepsilon^{2})E_{2}\big)(\mathbf{u}-\mathbf{u}^{(0)}),
\end{equation}
where
\begin{displaymath}
E_{1}=(I-MA)^{T}A(I-MA) \quad \text{and} \quad E_{2}=(I-MA)^{T}A^{\frac{1}{2}}\Pi A^{\frac{1}{2}}(I-MA).
\end{displaymath}

Observe that
\begin{displaymath}
E_{1}\succeq 0, \quad E_{2}\succeq 0, \quad \text{and} \quad E_{1}-E_{2}\succeq 0.
\end{displaymath}
We next determine upper bounds for $\|\mathbf{u}-\mathbf{u}^{(0)}\|_{E_{1}}^{2}$, $\|\mathbf{u}-\mathbf{u}^{(0)}\|_{E_{2}}^{2}$, and their difference.
\begin{itemize}[leftmargin=0.85cm]

\item Since
\begin{displaymath}
E_{1}=A^{\frac{1}{2}}\big(I-A^{\frac{1}{2}}\overline{M}A^{\frac{1}{2}}\big)A^{\frac{1}{2}},
\end{displaymath}
we obtain
\begin{displaymath}
\|\mathbf{u}-\mathbf{u}^{(0)}\|_{E_{1}}^{2}\leq\bigg(1-\min_{\mathbf{v}\in\mathcal{R}(A)\backslash\{0\}}\frac{\mathbf{v}^{T}A^{\frac{1}{2}}\overline{M}A^{\frac{1}{2}}\mathbf{v}}{\mathbf{v}^{T}\mathbf{v}}\bigg)\|\mathbf{u}-\mathbf{u}^{(0)}\|_{A}^{2}.
\end{displaymath}
A straightforward analysis yields
\begin{equation}\label{E1-est}
\|\mathbf{u}-\mathbf{u}^{(0)}\|_{E_{1}}^{2}\leq\big(1-\lambda_{n-r+1}(\widetilde{M}A)\big)\|\mathbf{u}-\mathbf{u}^{(0)}\|_{A}^{2},
\end{equation}
where we have used the fact
\begin{displaymath}
\lambda(\overline{M}A)=\lambda(\widetilde{M}A).
\end{displaymath}

\item It is easy to see that
\begin{displaymath}
E_{2}=A^{\frac{1}{2}}\big(I-A^{\frac{1}{2}}M^{T}A^{\frac{1}{2}}\big)\Pi\big(I-A^{\frac{1}{2}}MA^{\frac{1}{2}}\big)A^{\frac{1}{2}}.
\end{displaymath}
In light of~\eqref{maxeig-ide}, we have
\begin{equation}\label{E2-est}
\|\mathbf{u}-\mathbf{u}^{(0)}\|_{E_{2}}^{2}\leq(1-\delta_{\rm TG})\|\mathbf{u}-\mathbf{u}^{(0)}\|_{A}^{2}.
\end{equation}

\item Note that
\begin{align*}
E_{1}-E_{2}&=(I-MA)^{T}A^{\frac{1}{2}}(I-\Pi)A^{\frac{1}{2}}(I-MA)\\
&=(I-MA)^{T}A^{\frac{1}{2}}(I-\Pi)^{T}(I-\Pi)A^{\frac{1}{2}}(I-MA)\\
&=(I-MA)^{T}(I-\Pi_{A})^{T}A(I-\Pi_{A})(I-MA)\\
&=E_{\rm TG}^{T}AE_{\rm TG}.
\end{align*}
By~\eqref{iteration} and~\eqref{TG-est}, we have
\begin{equation}\label{E1-E2-est}
\|\mathbf{u}-\mathbf{u}^{(0)}\|_{E_{1}-E_{2}}^{2}=\|\mathbf{u}-\mathbf{u}_{\rm TG}\|_{A}^{2}\leq\rho^{2}\|\mathbf{u}-\mathbf{u}^{(0)}\|_{A}^{2}.
\end{equation}

\end{itemize}

According to~\eqref{iTG-est0}, we deduce that
\begin{displaymath}
\|\mathbf{u}-\mathbf{u}_{\rm ITG}\|_{A}^{2}\leq\|\mathbf{u}-\mathbf{u}^{(0)}\|_{E_{1}-E_{2}}^{2}+\varepsilon^{2}\|\mathbf{u}-\mathbf{u}^{(0)}\|_{E_{2}}^{2},
\end{displaymath}
which, combined with~\eqref{E2-est} and~\eqref{E1-E2-est}, leads to
\begin{equation}\label{iTG-est1}
\|\mathbf{u}-\mathbf{u}_{\rm ITG}\|_{A}\leq\sqrt{\rho^{2}+\varepsilon^{2}(1-\delta_{\rm TG})}\|\mathbf{u}-\mathbf{u}^{(0)}\|_{A}.
\end{equation}
In addition, we get from~\eqref{iTG-est0} that
\begin{displaymath}
\|\mathbf{u}-\mathbf{u}_{\rm ITG}\|_{A}^{2}\leq(1-\varepsilon^{2})\|\mathbf{u}-\mathbf{u}^{(0)}\|_{E_{1}-E_{2}}^{2}+\varepsilon^{2}\|\mathbf{u}-\mathbf{u}^{(0)}\|_{E_{1}}^{2},
\end{displaymath}
which, together with~\eqref{E1-est} and~\eqref{E1-E2-est}, yields
\begin{equation}\label{iTG-est2}
\|\mathbf{u}-\mathbf{u}_{\rm ITG}\|_{A}\leq\sqrt{\rho^{2}+\varepsilon^{2}\big(1-\rho^{2}-\lambda_{n-r+1}(\widetilde{M}A)\big)}\|\mathbf{u}-\mathbf{u}^{(0)}\|_{A}.
\end{equation}
The desired estimate then follows immediately by combining~\eqref{iTG-est1} and~\eqref{iTG-est2}.
\end{proof}

\begin{remark}
Recall that
\begin{displaymath}
\|\mathbf{u}-\mathbf{u}_{\rm TG}\|_{A}\leq\|E_{\rm TG}\|_{A}\|\mathbf{u}-\mathbf{u}^{(0)}\|_{A},
\end{displaymath}
where $\|E_{\rm TG}\|_{A}<1$. For the sake of practicality, we have made a more general assumption in Theorem~\ref{thm:iTG}. Another main assumption made in Theorem~\ref{thm:iTG} is~\eqref{accu-Bc}, which, in fact, characterizes the relative accuracy of the coarse solver $\mathscr{B}_{\rm c}\llbracket\cdot\rrbracket$.
\end{remark}

\begin{remark}
Since
\begin{displaymath}
\lambda(\widetilde{M}A)=\lambda\big(A^{\frac{1}{2}}\widetilde{M}A^{\frac{1}{2}}\big)\subset[0,1] \quad \text{and} \quad \rank\big(A^{\frac{1}{2}}\widetilde{M}A^{\frac{1}{2}}\big)\leq\rank\big(A^{\frac{1}{2}}\big)=r,
\end{displaymath}
it follows that
\begin{displaymath}
\lambda_{n-r+1}(\widetilde{M}A)>0\Longleftrightarrow\rank\big(A^{\frac{1}{2}}\widetilde{M}A^{\frac{1}{2}}\big)=r\Longleftrightarrow\mathcal{N}\big(A^{\frac{1}{2}}\widetilde{M}A^{\frac{1}{2}}\big)=\mathcal{N}(A),
\end{displaymath}
which will be satisfied if $\widetilde{M}\succeq 0$ and $\mathcal{N}(\widetilde{M})\cap\mathcal{R}(A)=\{0\}$.
\end{remark}

\begin{remark}
From Algorithms~\ref{alg:TG} and~\ref{alg:iTG}, we have
\begin{displaymath}
\|\mathbf{u}_{\rm TG}-\mathbf{u}_{\rm ITG}\|_{A}=\|P(\mathbf{e}_{\rm c}-\hat{\mathbf{e}}_{\rm c})\|_{A}=\|\mathbf{e}_{\rm c}-\hat{\mathbf{e}}_{\rm c}\|_{A_{\rm c}},
\end{displaymath}
i.e., the perturbation bound of two-grid solution is equal to that of coarse solution. Thus, the accuracy of two-grid solver can be controlled by that of coarse solver.
\end{remark}

Let $\rho_{\rm TG}$ and $\rho_{\rm ITG}$ denote the asymptotic convergence factors of exact and inexact two-grid methods, respectively. The estimate~\eqref{iTG-est} yields
\begin{displaymath}
\rho_{\rm ITG}\leq\sqrt{\rho_{\rm TG}^{2}+\varepsilon^{2}\big(1-\max\big\{\rho_{\rm TG}^{2}+\lambda_{n-r+1}(\widetilde{M}A),\,\delta_{\rm TG}\big\}\big)},
\end{displaymath}
which implies that
\begin{equation}\label{iTG-est-asym}
\rho_{\rm ITG}\leq\sqrt{\rho_{\rm TG}^{2}+\varepsilon^{2}\big(1-\rho_{\rm TG}^{2}\big)}.
\end{equation}
If $\rho_{\rm TG}$ can be bounded uniformly from above ($\rho_{\rm TG}\leq C$ for some constant $C<1$), then
\begin{displaymath}
\rho_{\rm ITG}\leq\sqrt{\varepsilon^{2}+(1-\varepsilon^{2})\rho_{\rm TG}^{2}}\leq\sqrt{\varepsilon^{2}+(1-\varepsilon^{2})C^{2}}<1 \quad\forall\,\varepsilon\in[0,1),
\end{displaymath}
that is, for any fixed $\varepsilon\in[0,1)$, $\rho_{\rm ITG}$ can also be bounded uniformly.

In addition, we deduce from~\eqref{iTG-est-asym} that
\begin{displaymath}
\rho_{\rm ITG}\leq\rho_{\rm TG}\sqrt{1+\frac{\varepsilon^{2}\big(1-\rho_{\rm TG}^{2}\big)}{\rho_{\rm TG}^{2}}}.
\end{displaymath}
This may serve as a guide for designing some practical algorithms. For example, if $\rho_{\rm TG}=0.3$, then $\rho_{\rm ITG}\leq 0.4$ can be achieved by setting $\varepsilon=\frac{1}{\sqrt{13}}$.

\subsection{Numerical experiments}

In order to intuitively illustrate the estimate~\eqref{iTG-est-asym}, we next provide a numerical example.

Consider solving the 2D Poisson's equation with homogeneous Neumann boundary conditions on the unit square $\Omega=(0,1)\times(0,1)$:
\begin{equation}\label{Poisson}
\left\{
\begin{aligned}
-\Delta u&=f \quad \text{in $\Omega$},\\
\frac{\partial u}{\partial n}&=0 \quad \text{on $\partial\Omega$}.
\end{aligned}
\right.
\end{equation}
The source term $f=f(x,y)$ should satisfy the following \textit{compatibility condition}:
\begin{displaymath}
\int_{\Omega}f(x,y)\,\mathrm{d}x\mathrm{d}y=0.
\end{displaymath}
To solve the problem~\eqref{Poisson} numerically, we partition the domain $\Omega$ uniformly and set
\begin{displaymath}
(x_{i},y_{j})=(ih,jh) \quad \forall\,i,j=0,1,\ldots,m+1,
\end{displaymath}
where $h=1/(m+1)$. A common way to discretize the problem~\eqref{Poisson} involves the following \textit{ghost points}:
\begin{align*}
(x_{-1},y_{j})&=(-h,jh) \quad j=0,\ldots,m+1;\\
(x_{m+2},y_{j})&=(1+h,jh) \quad j=0,\ldots,m+1;\\
(x_{i},y_{-1})&=(ih,-h) \quad i=0,\ldots,m+1;\\
(x_{i},y_{m+2})&=(ih,1+h) \quad i=0,\ldots,m+1;
\end{align*}
see the yellow points in Figure~\ref{Fig:grid}.

\begin{figure}[!htbp]
\centering
\includegraphics[width=1.0\textwidth]{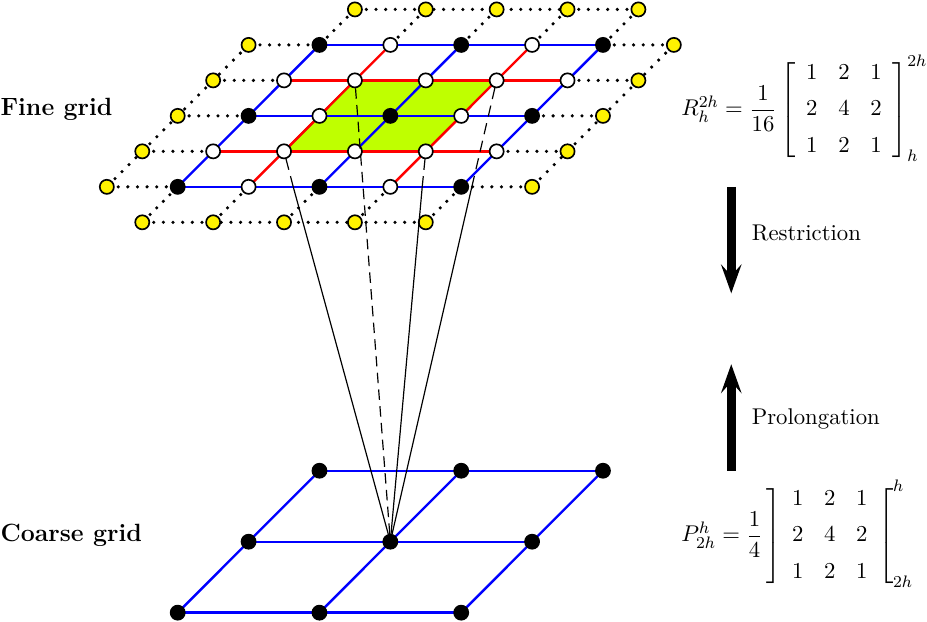}
\caption{\parbox{0.8\linewidth}{\small The restriction and prolongation processes between fine and coarse grids}}
\label{Fig:grid}
\end{figure}

The purpose of introducing these ghost points is to produce approximations at the boundary points. With this extended grid, we can apply the standard five-point difference scheme to the Laplacian operator at both interior and boundary points. Moreover, these ghost points will be used to form central difference approximations to the normal derivatives. These choices lead to the following discrete equations:
\begin{displaymath}
\left\{
\begin{aligned}
4u_{i,j}-u_{i-1,j}-u_{i+1,j}-u_{i,j-1}-u_{i,j+1}&=h^{2}f_{i,j} \quad i,j=0,\ldots,m+1,\\
u_{1,j}-u_{-1,j}&=0 \quad j=0,\ldots,m+1,\\
u_{m+2,j}-u_{m,j}&=0 \quad j=0,\ldots,m+1,\\
u_{i,1}-u_{i,-1}&=0 \quad i=0,\ldots,m+1,\\
u_{i,m+2}-u_{i,m}&=0 \quad i=0,\ldots,m+1,
\end{aligned}
\right.
\end{displaymath}
where $u_{i,j}\approx u(x_{i},y_{j})$ and $f_{i,j}=f(x_{i},y_{j})$. Eliminating the ghost unknowns $u_{-1,j}$, $u_{m+2,j}$, $u_{i,-1}$, and $u_{i,m+2}$ yields
\begin{displaymath}
\left\{
\begin{aligned}
8u_{i,j}-2u_{i-1,j}-2u_{i+1,j}-2u_{i,j-1}-2u_{i,j+1}&=2h^{2}f_{i,j} \quad i,j=1,\ldots,m,\\
4u_{i,0}-u_{i-1,0}-u_{i+1,0}-2u_{i,1}&=h^{2}f_{i,0} \quad i=1,\ldots,m,\\
4u_{0,j}-2u_{1,j}-u_{0,j-1}-u_{0,j+1}&=h^{2}f_{0,j} \quad j=1,\ldots,m,\\
4u_{m+1,j}-2u_{m,j}-u_{m+1,j-1}-u_{m+1,j+1}&=h^{2}f_{m+1,j} \quad j=1,\ldots,m,\\
4u_{i,m+1}-u_{i-1,m+1}-u_{i+1,m+1}-2u_{i,m}&=h^{2}f_{i,m+1} \quad i=1,\ldots,m,\\
2u_{0,0}-u_{1,0}-u_{0,1}&=\frac{1}{2}h^{2}f_{0,0},\\
2u_{m+1,0}-u_{m,0}-u_{m+1,1}&=\frac{1}{2}h^{2}f_{m+1,0},\\
2u_{0,m+1}-u_{1,m+1}-u_{0,m}&=\frac{1}{2}h^{2}f_{0,m+1},\\
2u_{m+1,m+1}-u_{m,m+1}-u_{m+1,m}&=\frac{1}{2}h^{2}f_{m+1,m+1}.
\end{aligned}
\right.
\end{displaymath}
Let
\begin{align*}
B&=\diag(1,2,\ldots,2,1)\in\mathbb{R}^{(m+2)\times(m+2)},\\
C&={\rm tridiag}(-1,4,-1)-2\big(\mathbf{e}_{1}\mathbf{e}_{1}^{T}+\mathbf{e}_{m+2}\mathbf{e}_{m+2}^{T}\big)\in\mathbb{R}^{(m+2)\times(m+2)},\\
D&={\rm tridiag}(-1,0,-1)\in\mathbb{R}^{(m+2)\times(m+2)},\\
\bm{\omega}&=(1,2,\ldots,2,1)^{T}\in\mathbb{R}^{m+2},\\
\mathbf{f}_{0}&=h^{2}\big(f_{0,0},f_{1,0},\ldots,f_{m+1,0},\ldots,f_{0,m+1},f_{1,m+1},\ldots,f_{m+1,m+1}\big)^{T}\in\mathbb{R}^{(m+2)^{2}},
\end{align*}
where $\mathbf{e}_{1}$ and $\mathbf{e}_{m+2}$ denote the first and last columns of the identity matrix $I_{m+2}$, respectively. By ordering all $u_{i,j}$ lexicographically, we end up with the linear system
\begin{displaymath}
A\mathbf{u}=\mathbf{f},
\end{displaymath}
where
\begin{displaymath}
A=B\otimes C+D\otimes B \quad \text{and} \quad \mathbf{f}=\frac{1}{2}(\bm{\omega}\otimes\bm{\omega})\circ\mathbf{f}_{0}.
\end{displaymath}
The symbols $\otimes$ and $\circ$ denote the Kronecker and Hadamard products of two matrices, respectively.

Assume that $m\geq 1$ is an odd number. We choose the points $(x_{i},y_{j})$ with indices $(i,j)\in\{0,2,\ldots,m+1\}\times\{0,2,\ldots,m+1\}$ as coarse ones; see the black points in Figure~\ref{Fig:grid}. Main settings in our numerical experiments are listed as follows:
\begin{itemize}[leftmargin=0.85cm]

\item $f$ is simply chosen as zero;

\item $\mathbf{u}^{(0)}$ is chosen randomly and then fixed;

\item $M$ is chosen as the symmetrized Gauss--Seidel type;

\item $R=R_{0}\otimes R_{0}$, where 
\begin{displaymath}
R_{0}=\frac{1}{4}\begin{pmatrix}
2 & 1 &   &   &   &   &   &   & \\
& 1 & 2 & 1 &   &   &   &   & \\
&   &   &   & \ddots &   &   &   & \\
&   &   &   &        & 1 & 2 & 1 & \\
&   &   &   &        &   &   & 1 & 2
\end{pmatrix}\in\mathbb{R}^{\frac{m+3}{2}\times(m+2)};
\end{displaymath}

\item $P=4R^{T}$, which is based on the prolongation stencil shown in Figure~\ref{Fig:grid};

\item $\mathscr{B}_{\rm c}\llbracket\cdot\rrbracket$ is chosen as the standard conjugate gradient method.

\end{itemize}
The corresponding numerical results are listed in Table~\ref{tab:TG} below.

\renewcommand{\arraystretch}{1.08}

\begin{table}[h!]
	
\begin{center}
	
\begin{tabular}{|c|c|ccccc|}
\hline
\multirow{2}{*}{} & $h$ & $1/32$ & $1/64$ & $1/128$ & $1/256$ & $1/512$ \\ \cline{2-7}
\multirow{2}{*}{} & $\rho_{\rm TG}$ & $0.2231$ & $0.2238$ & $0.2236$ & $0.2238$ & $0.2238$ \\
\hline
\multirow{2}{*}{$\varepsilon=0.1$} & $\rho_{\rm ITG}$ & $0.2231$ & $0.2238$ & $0.2236$ & $0.2238$ & $0.2238$ \\
& Upper bound in~\eqref{iTG-est-asym} & $0.2435$ & $0.2441$ & $0.2439$ & $0.2441$ & $0.2441$ \\
\hline
\multirow{2}{*}{$\varepsilon=0.2$} & $\rho_{\rm ITG}$ & $0.2231$ & $0.2238$ & $0.2236$ & $0.2238$ & $0.2238$ \\
& Upper bound in~\eqref{iTG-est-asym} & $0.2963$ & $0.2968$ & $0.2966$ & $0.2968$ & $0.2968$ \\
\hline
\multirow{2}{*}{$\varepsilon=0.3$} & $\rho_{\rm ITG}$ & $0.2598$ & $0.2848$ & $0.2913$ & $0.2947$ & $0.2963$ \\
& Upper bound in~\eqref{iTG-est-asym} & $0.3678$ & $0.3682$ & $0.3681$ & $0.3682$ & $0.3682$ \\
\hline
\multirow{2}{*}{$\varepsilon=0.4$} & $\rho_{\rm ITG}$ & $0.3285$ & $0.3816$ & $0.3852$ & $0.3915$ & $0.3946$ \\
& Upper bound in~\eqref{iTG-est-asym} & $0.4492$ & $0.4495$ & $0.4494$ & $0.4495$ & $0.4495$ \\
\hline
\multirow{2}{*}{$\varepsilon=0.5$} & $\rho_{\rm ITG}$ & $0.4500$ & $0.4669$ & $0.4833$ & $0.4918$ & $0.4932$ \\ 
& Upper bound in~\eqref{iTG-est-asym} & $0.5360$ & $0.5363$ & $0.5362$ & $0.5363$ & $0.5363$ \\
\hline
\end{tabular}

\end{center}

\medskip

\caption{\small Asymptotic convergence factors and their upper bounds}

\label{tab:TG}

\end{table}

From Table~\ref{tab:TG}, one can observe that
\begin{itemize}[leftmargin=0.85cm]

\item $\rho_{\rm TG}$ is essentially independent of grid size, and $\rho_{\rm ITG}$ can be bounded uniformly for any fixed $\varepsilon\in[0,1)$;

\item Satisfactory convergence may be achieved even if low-precision coarse solvers are used (so it may not be necessary to solve the coarse problem very precisely in practice);

\item Our theory can provide a good prediction for the actual convergence factors.

\end{itemize}

\section{Conclusions} \label{sec:con}

Two-grid theory plays a fundamental role in the design and analysis of multigrid methods. In this paper, we establish a new two-grid convergence theory for SPSD systems, which does not require any additional conditions on the coefficient matrix. When the Moore--Penrose inverse of coarse-grid matrix is used as a coarse solver, we derive a succinct identity for characterizing the convergence factor of two-grid methods. More generally, we present a convergence estimate for two-grid methods with approximate coarse solvers. A numerical example is also provided to illustrate our theoretical estimate.

\section*{Acknowledgments}

This work was partially supported by the National Natural Science Foundation of China (Grant No.~12401479), the Natural Science Foundation of Jiangsu Province (Grant No.~BK20241257), and the Start-up Research Fund of Southeast University (Grant No.~RF1028623372).

\bibliographystyle{amsplain}

\end{document}